\documentclass[12pt]{article}
\usepackage{graphicx}
\usepackage{color}
\usepackage{a4}
\usepackage{amsfonts}
\usepackage{amssymb}
\usepackage{xcolor}
\usepackage{amsmath,amsthm}
\usepackage{amsmath}

\definecolor{ashgrey}{rgb}{0.7, 0.75, 0.81}
\newtheorem{theorem}{Theorem}

\newtheorem{definition}[theorem]{Definition}

\newtheorem{lemma}[theorem]{Lemma}

\newtheorem{proposition}[theorem]{Proposition}
\newtheorem{remark}[theorem]{Remark}

\def \func#1{\mathop{\rm #1}\nolimits}
\begin{document}   
\title{A specific model of Hilbert geometry on the unit disc}
\author{Charalampos Charitos, Ioannis Papadoperakis \and and Georgios
Tsapogas\footnote{Corresponding author} \\
Agricultural University of Athens }
\maketitle
\begin{abstract} 
	A new metric on the open 2-dimensional unit disk is defined making it a  geodesically complete metric space whose geodesic lines are precisely the Euclidean straight lines. Moreover, it is shown that the unit disk with this new metric is not isometric to any hyperbolic model of constant negative curvature, nor to any convex domain in $\mathbb{R}^2$ equipped with its Hilbert metric.\\
	\textit{{2020 Mathematics Subject Classification:} 51F99}
\end{abstract}

\section{Introduction}

On the 8th August 1900, at the Second International Congress of Mathematics
held in Paris, David Hilbert delivered a lecture entitled \textquotedblleft
The future problems of mathematics\textquotedblright , in which he presented
a collection of open problems. The fourth problem of the list can be stated
as follows: If $\Omega $ is a convex subset of a Euclidean space, find a
characterization of all metrics on $\Omega $ for which the Euclidean lines
are geodesics. We can put additional conditions on these metrics on $\Omega $
by requiring geodesic completeness and Euclidean lines being the
unique geodesics. These geometries with the extra requirements are
 of particular interest and they have been studied extensively.

Before Hilbert, Beltrami in \cite{Beltrami} had already shown that the
unit disc in the plane, with the Euclidean chords taken as geodesics of
infinite length, is a model of the hyperbolic geometry. However, Beltrami did
not give a formula for this distance, and this led Klein in \cite{Klein}
to express the distance in the unit disc in terms of the cross radio. 
Hilbert's fourth problem became a very active research
area and it was gradually realized that the discovery of all metrics satisfying Hilbert's problem was not plausible.
Consequently, each metric resolving Hilbert's problem defines a new
geometry worth to be studied. A very important class of such metrics,
defined by means of the cross ratio, are referred to as \textit{Hilbert
metrics} and play a central role in this research area.

Among the prominent mathematicians worked on the Hilbert's fourth problem, it
is worthy to mention Busemann and Pogorelov, see for instance \cite{Buseman}, 
\cite{Pogorelov1}, \cite{Pogorelov2}. The ideas of the latter to solve
Hilbert's fourth problem came from Busemann, who introduced 
integral geometry techniques to approach Hilbert's problem. Busemann's idea was to
consider for every two points $x$ and $y$ in a convex subset $\Omega $ of
the real projective space $RP^{n},$ the unique geodesic segment $[x,y]$
joining these points, and the subset of hyperplanes of $RP^{n}$ intersecting 
$[x,y]$ equipped with a  non negative measure having specific properties. In
dimension $2,$ Pogorelov's solution consisted in showing that every distance
between $x$ and $y$ satisfying Hilbert's problem is given by a metric $d(x,y)
$ constructed with the help of the measure constructed by Busemann on the
subset of hyperplanes mentioned above. There are generalizations of 
Pogorelov's theorem in greater dimensions and one may see in \cite{Papadopoulos}  a detailed discussion on Hilbert's fourth problem. 

However, Pogorelov's approach is very general and does not allow for further study of the geometry of these metrics. On the contrary, in the
present work a concrete new metric satisfying Hilbert's problem is defined without the use of cross ratio and its geometry is studied. More precisely, it is shown that this metric makes the open unit disk a geodesically complete metric space whose geodesics are of infinite length and are precisely the Euclidean lines. Moreover it is shown that this metric space is not isometric to any hyperbolic model of constant curvature nor to any convex domain in the plane equipped with its Hilbert metric. Finally, the natural Euclidean boundary of the unit disk is shown to coincide with the visual boundary with respect to the new metric, namely, with the set of equivalence classes of asymptotic geodesic rays.

\section{Definitions}

Let $f:\left( -1,1\right) \longrightarrow \mathbb{R}$ be the function

\begin{minipage}{9cm}\[f(t)= \frac{t}{1-|t|}\]\end{minipage}
\begin{minipage}{4.3cm}
	\includegraphics[scale=0.55]{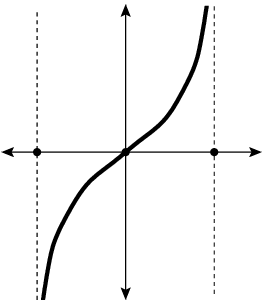}
	\begin{picture}(0,0)
		\put(-15,32){$\scriptstyle 1$}  
		\put(-39,32){$\scriptstyle 0$} 
		\put(-75,32){$\scriptstyle -1$}
	\end{picture}
\end{minipage}\\
and define a metric $d_{I}$ on the interval $\left( -1,1\right) $ by
\begin{equation*}
 d_{I}\left( s,t\right) = \left \vert f\left( s\right) -f\left( t\right) \right \vert .
\end{equation*}
Clearly, 
\begin{equation*}d_{I}\left( t,0\right)=d_{I}\left( 0,t\right) =\left \vert
\frac{t}{1-|t|}\right \vert=  \frac{1}{1-|t|}-1 ,\mathrm{\ \ }\forall t\in(-1,1)\end{equation*}
and 
\begin{equation*}
d_{I}\left( s,t\right) = f\left( |s|\right) +f\left( |t|\right)\mathrm{\ if\ } st<0.
\end{equation*}
Moreover, this metric has the following two properties:

\begin{enumerate}
\item[P1:] For $s,t\in(-1,1),$ $d_{I}\left( s,t\right) \longrightarrow \infty$ as $ t\longrightarrow -1$ or $1.$

\item[P2:] For $ q,s,t\in(-1,1)$ with $q<s<t$ we have $d_{I}\left( q,s\right)+d_{I}\left( s,t\right) =d_{I}\left( q,t\right). $
\end{enumerate}

We now define a metric $D$ on the open 2-dimensional unit disk $\mathbb{D}^{2}$ as
follows: consider the unit disk in the $xy-$plane and identify each ray with
an angle $\theta \in \left[ 0,2\pi \right] .$ For each such ray (eg. angle $%
\theta $) denote by $\Delta _{\theta }$ the diameter determined by that ray.
Observe that the Euclidean length of $\Delta _{\theta }$ is $2.$ For any
point $z$ in the disk, denote by $z_{\theta }$ its projection to the
diameter $\Delta _{\theta }.$ Let $x,y$ be two points in the disk.

For each $\theta \in \left[ 0,2\pi \right] ,$ set $d_{\theta}\left(
x,y\right) :=d_{I}\left( x_{\theta},y_{\theta}\right) $ where
the projection points
$x_{\theta},y_{\theta}$
in $\Delta \theta$
are identified with the corresponding points in the interval $(-1,1). $ Define
\begin{equation}
D\left( x,y\right) := \frac{1}{2} \int_{0}^{2\pi}d_{\theta}\left( x,y\right)
d\theta .\label{zerotwopi}
\end{equation}
Since, for every $\theta \in [0,\pi]$ the diameters $\Delta_{\theta}$ and $%
\Delta_{\theta +\pi}$ coincide, we have $d_{\theta}\left( x,y\right)= d_{\pi
+\theta}\left( x,y\right)$ for any two points $x,y.$ It follows that 
\begin{equation}
D\left( x,y\right) =\frac{1}{2}\int_{0}^{2\pi}d_{\theta}\left( x,y\right)
d\theta = \int_{0}^{\pi}d_{\theta}\left( x,y\right) d\theta
\label{zeropi}
\end{equation}

\begin{lemma}[Triangle Inequality]
\label{unique} Let $x,y,z$ be three points in the disk. \newline
(a) If the Euclidean segment $[x,z]$ in the unit disk contains the point $y$
then 
\begin{equation*}
D(x,z)= D(x,y)+D(y,z).
\end{equation*}
(b) If $y$ is a point not contained in $[x,z]$ then 
\begin{equation*}
D(x,z)< D(x,y)+D(y,z).
\end{equation*}
\end{lemma}

\begin{proof}
(a) The projection $y_{\theta} $ will be in the interior of the segment $\left[x_{\theta},z_{\theta}\right]\subset \Delta_{\theta}$ for all directions $\theta$ except in the case the direction  $\theta$ is perpendicular to the segment $[x,z]$ (in which case $x_{\theta}\equiv z_{\theta}$). By Property P2 
\[ d_{\theta}\left(  x,y\right)+d_{\theta}\left(  y,z\right)  =d_{\theta}\left(  x,z\right)\]
which shows the desired equality. 

\mbox{$\,$}\\[2mm] \raisebox{1.7cm}{(b)}  \begin{minipage}{8.3cm}Set $\theta_{yx}$ (resp. $\theta_{yz}$) to be the angle which, viewed as a ray, is perpendicular to the Euclidean segment $[y,x]$ (resp. $[y,z]$). We may assume that  $\theta_{yx}<\theta_{yz}.$ Then for every $\theta \notin [\theta_{yx}, \theta_{yz}]$ the point $y_{\theta}$ is contained in the segment $[x_{\theta} ,z_{\theta}]$  so that by Property P2 we have
\[d_{\theta} (x,z)= d_{\theta} (x,y) + d_{\theta} (y,z).\]
\end{minipage}\hfill
  \begin{minipage}{4.3cm}
\includegraphics[scale=0.85]{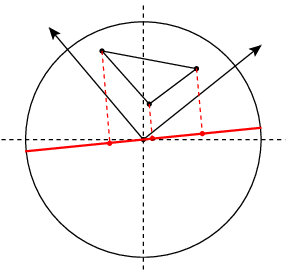}
  \begin{picture}(0,0)   
  \put(-13,90){$\scriptstyle \theta{yx}$}    
  \put(-113,100){$\scriptstyle \theta{yz}$}  
  \put(-12,58){$\color{red}\scriptstyle \Delta_{\theta}$}  
  \put(-42,45){$\color{red} z_{\theta}$} 
  \put(-82,43){$\color{red} x_{\theta}$}   
  \put(-62,44){$\color{red} y_{\theta}$} 
  \put(-42,83){$ z$} 
  \put(-82,92){$ x$} 
  \put(-62,74){$\scriptstyle y$}
 \end{picture}
  \end{minipage}\\[3mm]
  For $\theta \in [\theta_{yx}, \theta_{yz}]$ the point $y_{\theta}$ is not contained in the segment $[x_{\theta} ,z_{\theta}]$ which implies that 
  \[d_{\theta} (y,z)= d_{\theta} (y,x) + d_{\theta} (x,z)
  \textrm{\ \ or,\ \ }
  d_{\theta} (y,x)= d_{\theta} (y,z) + d_{\theta} (z,x)
  \]
  depending on whether $x_{\theta}\in [y_{\theta} ,z_{\theta}]$ or, 
  $z_{\theta}\in [y_{\theta} ,x_{\theta}].$ In the first case we  have 
  \[ d_{\theta} (x,z) <  d_{\theta} (y,z) \Longrightarrow 
  d_{\theta} (x,z) \lneqq d_{\theta} (y,z) + d_{\theta} (x,y)\]
  and the same strict inequality follows in the second case. This completes the proof of part (b).
\end{proof}

\begin{lemma}
\label{metric} The function $D\left( x,y\right)
=\int_{0}^{\pi}d_{\theta}\left( x,y\right) d\theta $ defines a metric on the
open unit disk.
\end{lemma}

\begin{proof}
To complete the proof that $D$ is a metric we only need to show that the integral 
\[ \int_{0}^{\pi}d_{\theta}\left(  x,y\right)  d\theta \] is finite. By Lemma \ref{unique}, it suffices to show that for any $x$ in the unit disk the integral 
$\int_{0}^{\pi}d_{\theta}\left(  O, x\right)  d\theta $ is finite.
Recall that the point $O$ in the diameter $\Delta_{\theta}$ is identified with  $0\in (-1,1)$ and observe that by definition of the function $f$ the orientation of the diameter $\Delta_{\theta}$ is irrelevant. Let $\| \cdot \|$ denote Euclidean length. Clearly, 
\begin{equation*}
 d_I (0,x_{\theta})\leq d_I (0,\| Ox \|)
\end{equation*}
We then have 
\[ D(O,x) = \int_{0}^{\pi} d_{\theta} (O,x)d\theta =
\int_{0}^{\pi} d_I (0,x_{\theta})d\theta \leq
\int_{0}^{\pi} d_I (0,\| Ox \|)d\theta  = \frac{\| Ox \|}{1-\| Ox \|}\pi .\]
\end{proof}
It is well known that a curve with endpoints $x,z$ is a geodesic segment with respect to 
a metric $d$ if and only if for every $y$ in the curve we have $%
d(x,y)+d(y,z)=d(x,z).$ It follows, by part (a) of Lemma \ref{unique} above,
that Euclidean lines in the unit disk are geodesics with respect to the
metric $D$ and part (b) shows that only the Euclidean lines are geodesics
with respect to  $D.$ Hence, we have the following

\begin{proposition}\label{basic}
 The metric space $\left( \mathbb{D}^{2},D\right) $ is a geodesic metric space whose geodesics are precisely the Euclidean lines in $\mathbb{D}^{2}.$
\end{proposition}

The following properties follow from the definition of the metric $D.$ 
\begin{proposition}\label{rotary}
(a)  Every Euclidean rotation $R_{\phi} : \mathbb{D}^{2} \longrightarrow \mathbb{D}^{2}, \phi\in [0,2\pi] ,$ centered at the origin is an isometry of the metric space  $\left( \mathbb{D}^{2},D\right). $\\
 (b) Every Euclidean reflection $Q_{\phi} : \mathbb{D}^{2} \longrightarrow \mathbb{D}^{2}$ with respect to a line forming an angle $\phi\in[0,2\pi]$ with the $x-$axis is an  isometry of the metric space  $\left( \mathbb{D}^{2},D\right). $
\end{proposition}
\begin{proof}
 As $\int_{0}^{2\pi}d_{\theta}\left( x,y\right)d\theta =  \int_{\phi}^{2\pi +\phi}d_{\theta}\left( x,y\right)d\theta$ part (a) follows.

 For (b), it suffices, by (a), to show that the Euclidean reflection $R_0$ with respect to the $x-$ axis is an isometry. Clearly, for arbitrary $x,y\in \mathbb{D}^{2}$ and for every $\theta\in [0,2\pi]$ we have
 \[ d_{\theta} \left( x,y  \right) = d_{2\pi - \theta} \left( R_0(x),R_0(y)  \right)  \]
 which implies that $ D(x,y)=D\left( R_0(x),R_0(y) \right).$
\end{proof}

We now proceed to show that the metric space  $\left( \mathbb{D}^{2},D\right)  $ is geodesically complete, that is, every geodesic segment extends uniquely to a geodesic line of infinite length.

If $\xi$ a point on the boundary, 
$d_{\theta }\left( O,\xi \right) $ can be defined via projections as before and it is a positive real for all $\theta,$
except for a single value in $\left[ 0,\pi \right) .$ Thus, the integral $%
\int_{0}^{\pi}d_{\theta}\left( O, \xi \right) d\theta$ makes sense and we have

\begin{lemma}\label{rayINF}
If $O$ is the center of the unit disk and $\xi$ a point on the boundary, the
integral $\displaystyle\int_{0}^{\pi}d_{\theta}\left( O,\xi\right) d\theta$
is not bounded.
\end{lemma}

\begin{proof}
\mbox{$\,$}\\[2mm]  \begin{minipage}{9.3cm}
As above, the point $O$ in the diameter $\Delta_{\theta}$ is identified with  $0\in (-1,1)$ and, if $\theta_{\xi}$ is the angle determined by 
$\xi ,$ for all $\theta\neq \theta_{\xi}$   we have 
\begin{align*}
d_{\theta}\left(  O,\xi\right)
&=  
d_I\left( 0, \|O\xi \| \cos \left( \theta-\theta_{\xi} \right) \right) 
= \left| f \left( \cos \left( \theta-\theta_{\xi} \right) \right)\right|
\\
& =\left|\frac{\cos\left( \theta-\theta_{\xi} \right)}{1-\left| \cos\left( \theta-\theta_{\xi} \right) \right|}\right|
=\frac{1}{1-\left| \cos\left( \theta-\theta_{\xi} \right) \right|} -1 .
\end{align*} 
\end{minipage}\hfill
  \begin{minipage}{4.3cm} 
\includegraphics[scale=0.85]{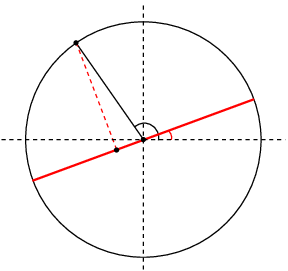} 
  \begin{picture}(0,0)   
  \put(-94,97){$\xi$}    
  \put(-12,58){$\color{red}\scriptstyle \Delta_{\theta}$}  
  \put(-45,56){$\color{red} _{\theta}$} 
  \put(-79,39){$\color{red} \xi_{\theta}$}   
    \put(-61,46){$\scriptstyle  O$} 
  \put(-62,61){$ \theta _{\xi}$} 
 \end{picture}
  \end{minipage}\\[3mm] 
Then the integral $\int_{0}^{\pi}d_{\theta}\left(  O,\xi\right)  d\theta$ equals

\begin{align*}\hskip-4mm\frac{1}{2}
	 \int_{0}^{2\pi}& \left[\frac{1}{1- \left| \cos \left( \theta-\theta_{\xi} \right)  \right| } -1\right] d\theta
	= \frac{1}{2}	 \int_{0}^{2\pi} \left[\frac{1}{1- \left| \cos \theta   \right| } -1\right] d\theta = \\
	&\hskip22mm = \frac{1}{2}	4 \int_{0}^{\pi /2} \left[\frac{1}{1- \cos  \theta  } -1\right] d\theta 
  > 2 
 \int_{0}^{\pi/2} \left[\frac{1}{ \theta } -1\right] d\theta = \infty.
\end{align*}

\end{proof}The above Lemma permits us to say that the $D-$length of a
Euclidean ray (or, diameter) is infinite. In a similar manner, 
if $\xi , \eta$ are two points on the boundary, $d_{\theta }\left( \xi, \eta \right) $ is a positive real for all $\theta,$
except for two values of $\theta$ in $\left[ 0,\pi \right) .$ Thus, the integral $
\int_{0}^{\pi}d_{\theta}\left( \xi,\eta \right) d\theta$ makes sense and we have
\begin{lemma}\label{lineINF}
If $\xi, \eta$ are two points on the boundary, the
integral $\displaystyle\int_{0}^{\pi}d_{\theta}\left(\xi,\eta\right) d\theta$ is not bounded.
\end{lemma}
\begin{proof}
By Proposition \ref{rotary}, we may assume that the geodesic line determined by $\xi,\eta$ is perpendicular to the $x-$axis with intersection point, say, $A.$ It suffices to show that  $\displaystyle\int_{0}^{\pi}d_{\theta}\left(A, \xi\right) d\theta$ is not bounded.

Let $\theta_{\xi}$ be the angle formed by the $x-$axis and the geodesic ray joining $O$ with $\xi  .$ For all $\theta \in \left[ 0,\pi/2\right]\setminus \left\{ \theta_{\xi}\right\}$ we have 
\[ d_{\theta}\left( O,\xi \right) =d_{\theta}\left( O, A \right) +d_{\theta}\left( A,\xi \right) \] 
and for all $\theta \in \left[ 0,\pi/2\right]$ we have
\[ d_{\theta}\left( O,\xi \right)< d_{\theta}\left( A,\xi \right) < d_{\theta}\left( O, A \right) +d_{\theta}\left( A,\xi \right). \]
Therefore, the triangle inequality 
\[  d_{\theta}\left( O,\xi \right) \leq d_{\theta}\left( O, A \right) +d_{\theta}\left( A,\xi \right)  \] holds for all $\theta \in \left[ 0,\right]\setminus \left\{ \theta_{\xi}\right\}.$
By Lemma \ref{rayINF},  $\int_{0}^{\pi}d_{\theta}\left( O, \xi\right) d\theta$ is not bounded and $\int_{0}^{\pi}d_{\theta}\left( O, A\right) d\theta = D(O,A)$ is a positive real, thus, 
\begin{equation*}
 \displaystyle\int_{0}^{\pi}d_{\theta}\left(A, \xi\right) d\theta 
 \end{equation*}
cannot be bounded.
\end{proof}
\begin{remark}
	There exists a large family of metrics making the open unit disk a geodesic metric space satisfying Propositions \ref{basic}, \ref{rotary} and Lemmata \ref{rayINF}, \ref{lineINF}. In fact, for every strictly increasing function 
	$g: (-1,1)\rightarrow \mathbb{R}$ which satisfies 
	\[ \lim_{t\rightarrow 1} \int_{0}^{2\pi} \left| g\left( t \cos\theta  \right) \right|  d\theta =\infty\]
	we may apply the above construction using $g$ instead of $f$ and obtain a metric on the open unit disk with the above mentioned properties.
\end{remark} 
\section{Further properties of $\left( \mathbb{D}^{2},D\right) $}
For any $\kappa >0,$ let $\mathbb{H}_{-\kappa ^{2}}^{2}$ denote the standard hyperbolic model of constant negative curvature on the open unit disk with distance function $d_{\kappa }.$ For a convex domain $U$ in $\mathbb{R}^{2}$ denote by $d_{\mathcal{H}}$ the Hilbert metric for which we refer the reader to \cite[Ch.5, Section 6]{Pap}.

\begin{theorem}\label{eight}
For all $\kappa >0,$ the metric spaces $\left( \mathbb{D}^{2},D\right) $ and 
$\left( \mathbb{H}_{-\kappa ^{2}}^{2},d_{\kappa }\right) $ are not isometric. Moreover, for any convex domain $U$ in $\mathbb{R}^{2}$ equipped with the Hilbert metric $d_{\mathcal{H}}$  the metric spaces $\left( \mathbb{D}^{2},D\right) $ and 
$\left( U,d_{\mathcal{H}}\right) $ are not isometric.
\end{theorem}

\begin{proof} Assume $F_{\kappa}:\mathbb{D}^{2}\longrightarrow \mathbb{H}_{-\kappa
^{2}}^{2}$ is such an isometry which, by homogeneity of $\mathbb{H}_{-\kappa
^{2}}^{2},$ we may assume that $F_{\kappa}$ preserves the center $O$ of the disk. Moreover, as the image of the $x-$axis under  $F_{\kappa}$ is a line in  $\mathbb{H}_{-\kappa^{2}}^{2}$ containing the origin, by composing  $ F_{\kappa}$  with a rotation in $\mathbb{H}_{-\kappa^{2}}^{2},$ we may assume that $F_{\kappa}$ preserves the $x-$axis. We next show that $F_{\kappa}$ necessarily preserves the $y-$axis as well, by showing that the image of the  $y-$axis under $F_{\kappa}$ is a line perpendicular to the $x-$axis. To see this, for any $%
x\in \left( 0,1\right) $ consider the quadrilateral $XYZW$ where $X=(x,0),$ $
Y=(0,x),$ $Z=\left( -x,0\right) $ and $W=\left( 0,-x\right) .$ Clearly, $XYZW$ is a square with respect to  the metric $D$ and, hence, so must be its image under $F_{\kappa}.$ The
geodesic segments $\left[ F_{\kappa}(X),F_{\kappa}\left( Z\right) \right] $ and $\left[
F_{\kappa}(Y),F_{\kappa}\left( W\right) \right] $ intersect at $O=F_{\kappa}(O)$ which is the midpoint for both segments. Moreover, these segments must form a right angle at $O$, otherwise the quadrilateral $F_{\kappa}(X)F_{\kappa}\left( Y\right) F_{\kappa}(Z)F_{\kappa}\left(W\right) $ cannot be a square.

Define the function $h:\left[ 0,\infty \right] \longrightarrow \mathbb{R}$
where $h(b),$ for $b\in \left[ 0,\infty \right) ,$  is the $d_{1}-$length of
the height of the right angle hyperbolic triangle in $\mathbb{H}_{-1}^{2}$
with side lengths equal to $b.$ For $b=\infty ,$ $h\left( \infty \right) $
is the $d_{1}-$length of the height of the right angle ideal hyperbolic
triangle in $\mathbb{H}_{-1}^{2}$ with vertices $O,(1,0)$ and $(0,1).$ By
elementary calculations, $h\left( \infty \right) $ is the $d_{1}-$length of
the segment with endpoints $O$ and $\left( 1-\sqrt{2}/2,1-\sqrt{2}/2\right) .
$ Thus its Euclidean length is $\sqrt{2}-1$ and 
\begin{equation} 
h\left( \infty \right) =\log \frac{1+\left( \sqrt{2}-1\right) }{1-\left( 
\sqrt{2}-1\right) }=\log \left( 1+\sqrt{2}\right) \approx 0.8813735870
\label{hofinfty}
\end{equation}%
For $b\in \left[ 0,\infty \right) $ let $B$ be the point on the positive $x-$axis so that $OB$ has hyperbolic length $b$ and denote by $C$ the trace from the origin of the height of the right angle hyperbolic triangle in $\mathbb{H}_{-1}^{2}$
with side lengths equal to $b.$ Then the triangle $\triangle (OCB)$ has a right angle at $C$ and $\widehat{COB}=\pi/4.$ Using the formula
\[\cos \frac{\pi }{4}=\frac{\tanh \left( h\left( b\right) \right) }{\tanh \left( b\right) }\]
we find
\begin{equation}
h\left( b\right) =\frac{1}{2}\log \frac{1+(\sqrt{2}/2)\tanh b}{1-(\sqrt{2}%
/2)\tanh b}  \label{hofb}
\end{equation}
We next define an analogous function $h^{D}$ for the metric space $\left( 
\mathbb{D}^{2},D\right) $ with a different domain%
\begin{equation*}
h^{D}:\left[ 0,1\right] \longrightarrow \mathbb{R}
\end{equation*}%
where for $x\in \left[ 0,1 \right) ,$ $h^{D}(x)$ is the $D-$length of
the height of the right angle geodesic triangle $T_{x}$ with vertices $%
O,(x,0)$ and $\left( 0,x\right) .$ For $x=1,$ $h^{D}\left( 1\right) $ is the length of the 
height  of the corresponding ideal geodesic triangle. As all rotations
of the unit disk are isometries of $\left( \mathbb{D}^{2},D\right) $ we have 
\begin{equation}
h^{D}\left( x\right) =D\left( O,\frac{\sqrt{2}}{2}x\right) \mathrm{\ \ for\
all\ }x\in \lbrack 0,1].  \label{hDx}
\end{equation}
We will explicitly compute the $D-$lengths of the heights of the triangles $T_{\sqrt{2}/2}$ and $T_{1}$
and compare them with the corresponding $d_{k}-$lengths of the heights of the triangles $F_{\kappa}\left( T_{\sqrt{2}/2}\right) $ and $F_{\kappa}\left( T_{1}\right) .$ The comparison of the lengths of the heights of $T_{1}$ and $F_{\kappa}\left( T_{1}\right) $ will suffice to reach a contradiction for the case $\kappa =1.$ The triangles  $T_{\sqrt{2}/2}$ and $F_{\kappa}\left( T_{\sqrt{2}/2}\right) $ are deployed in order to reach a contradiction for all $\kappa .$

We first compute the  $D-$lengths of the heights of the triangles $T_{\sqrt{2}/2}$ and $T_{1} .$\\
For the triangle $T_{\sqrt{2}/2},$ by (\ref{hDx}) and using the easily verified fact that the
derivative with respect to  $\theta $ of the function $\frac{4\tan^{-1}\left( \frac{1}{\sqrt{3}}\tan\left( \frac{\theta}{2} \right)\right)}{\sqrt{3}} $ is $\frac{2}{2-\cos x },$ we have

\begin{align}
h^{D}\left( \frac{\sqrt{2}}{2}\right) & =D\left( O,\frac{\sqrt{2}}{2} \,\, \frac{\sqrt{2}}{2} \right)
= D\left( O,\frac{1}{2}\right)
=\int_{0}^{\pi }\left[ \frac{1}{1-\frac{1}{2}\left\vert \cos \theta
\right\vert }-1\right] d\theta  \notag  \\
& =\int_{0}^{\pi /2}\left[ \frac{1}{1-\frac{1}{2}\cos \theta }-1\right]
d\theta +\int_{\pi /2}^{\pi }\left[ \frac{1}{1+\frac{1}{2}\cos \theta }-1%
\right] d\theta   \notag \\
& =2 \left[ \frac{4\tan ^{-1}\left( \sqrt{3}\tan \left( \frac{\theta }{2}%
\right) \right) }{\sqrt{3}}-\theta \right] _{\theta =0}^{\theta =\pi /2}
\notag \\
& =2\left[ \frac{8\sqrt{3}-9}{18}\pi \right]  =\frac{8\sqrt{3}-9}{9}\pi \hskip0.2cm\approx 1.695205651 
\label{hDr2line1}
\end{align}%
A similar computation, using again the easily verified fact that the
derivative with respect to  $\theta $ of the function $2\sqrt{2}\tan ^{-1}\left( \frac{%
\left( \sqrt{2}+2\right) \tan \left( \frac{\theta }{2}\right) }{\sqrt{2}}%
\right) $ is $\frac{2+\sqrt{2}}{-(1+\sqrt{2})\cos \theta +\sqrt{2}+2},$
shows 
\begin{equation}
h^{D}\left( 1\right) =D\left( O,\frac{\sqrt{2}}{2}\right) =\frac{3\sqrt{2}-2}{2}\pi \approx 3.522731754  \label{hD1}
\end{equation}
The above calculation along with (\ref{hofinfty}) shows that
\begin{equation}
h^{D}\left( 1\right) =\frac{3\sqrt{2}-2}{2}\pi  \neq \log \left( 1+\sqrt{2}\right)= h\left( \infty \right)
\label{refsugg}
\end{equation}
and, thus, the triangles $T_{1}$ and $F_{1}\left( T_{1}\right) $ cannot be isometric. It follows that $F_{\kappa}$ cannot be an isometry in the case $\kappa =1.$

Before proceeding with the general case, we compute the $d_{1}-$length of the height of the  triangle $F\left(T_{\sqrt{2}/2}\right) .$ This triangle, being isometric to $T_{\sqrt{2}/2},$ has side lengths $D\left( O,\frac{\sqrt{2}}{2}\right) $ and, using (\ref{hofb}), its height has $d_{1}-$length
\begin{align}
h\left( D\left( O,\frac{\sqrt{2}}{2}\right) \right)
&
 =h\left( \frac{3\sqrt{2}-2}{2}\pi \right) \nonumber\\
& =\frac{1}{2}\log \frac{1+(\sqrt{2}/2)\tanh \left( \frac{3
\sqrt{2}-2}{2}\pi \right) }{1-(\sqrt{2}/2)\tanh \left( \frac{3\sqrt{2}-2}{2}
\pi \right) }\approx 0.8789154496  \label{lastH}
\end{align}

We proceed now with the general case. Recall that geodesic lines in $\mathbb{H}_{-\kappa
^{2}}^{2}$ and $\mathbb{H}_{-1}^{2}$ coincide as subsets of $\mathbb{D}^{2}$
and lengths are multiplied by $\kappa .$ Therefore,
$F_{\kappa} \left( T_{1}\right)  = F_{1} \left( T_{1}\right) $ and
the $d_{\kappa}-$length of the height of the triangle $F_{\kappa} \left( T_{1}\right) $ is equal to $\kappa \,h\left( \infty \right). $ By (\ref{refsugg}), it follows that $F_{\kappa}$ can be an isometry only for the model $\mathbb{H}_{-\kappa^{2}_0}^{2}$ where
\begin{equation}
\kappa_0 =\frac{h^{D}\left( 1\right)}{h\left( \infty \right)}
 =\frac{3\sqrt{2}-2}{2 \log \left( 1+\sqrt{2}\right) }\pi >3.
\label{refsugg2}
\end{equation}
To rule out this last case we compare the $D-$length of the height of the triangle $T_{\sqrt{2}/2}$ (computed in (\ref{hDr2line1}) above) with the $d_{\kappa_0}-$length of the height of the  triangle
$F_{\kappa_0}\left( T_{\sqrt{2}/2}\right) .$ As before, the triangles $F_{1}\left( T_{\sqrt{2}/2}\right) $ and $F_{\kappa_0}\left( T_{\sqrt{2}/2}\right) $ coincide as sets and the $d_{\kappa_0}-$length of its height is its $d_{1}-$length (computed in (\ref{lastH}) above) multiplied by $\kappa_0 .$ Therefore, if the triangles  $T_{\sqrt{2}/2}$ and  $F_{\kappa_0}\left( T_{\sqrt{2}/2}\right) $ were isometric,
$\kappa_0$ would have to  satisfy
\[ h^{D}\left( \frac{\sqrt{2}}{2}\right) = \kappa_0 \,\,  h\left( D\left( O,\frac{\sqrt{2}}{2}\right) \right)\]
which is impossible because $\kappa_0 >3$ (see (\ref{refsugg2})) and the ratio of
$h^{D}\left( \frac{\sqrt{2}}{2}\right) $ and $h\left( D\left( O,\frac{\sqrt{2}}{2}\right) \right)$ is, by (\ref{hDr2line1}) and (\ref{lastH}), equal to
\[ \frac
{h^{D}\left( \frac{\sqrt{2}}{2}\right)}
{h\left( D\left( O,\frac{\sqrt{2}}{2}\right) \right)}=
\frac{\frac{8\sqrt{3}-9}{9}\pi }{h\left( \frac{3\sqrt{2}-2}{2}\pi \right)}
 \approx \frac{1.695205651 }{0.8789154496} <3.
\]
This completes the proof for arbitrary curvature.

To see that $\left( \mathbb{D}^{2},D\right) $ is not isometric to any convex domain $U$ equipped with its Hilbert metric $d_{\mathcal{H}}$, we will use a result of Busemann and Kelly (see \cite[\S 29.2]{BusKel}) which states the following: let $U$ be a bounded open convex domain in $\mathbb{R}^2 .$ Reflections with respect to all lines in $U$ through one fixed point exist if and only if $U$ is the interior of an ellipse. 

The reflection assumption holds for the metric space $\left( \mathbb{D}^{2},D\right) ,$ see Proposition \ref{rotary}. If $\left( \mathbb{D}^{2},D\right) $ were isometric to some $\left( U,d_{\mathcal{H}}\right) $ then the same reflection assumption would hold for $U$ and, hence, $U$ would have to be the interior of an ellipse making $d_{\mathcal{H}}$ the hyperbolic metric. As shown above this cannot be the case.
\end{proof}
We next restrict our attention to geodesic rays, that is, isometric maps $[0,\infty) \rightarrow \mathbb{D}^2.$
\begin{definition}
Two geodesic rays $r_1 , r_2 : [0,\infty) \rightarrow \mathbb{D}^2$ in the geodesic metric space $\left( \mathbb{D}^{2},D\right) $  are called asymptotic if the distance function $t\rightarrow D(r_1(t),r_2 (t))$ is bounded.
\end{definition}

\begin{remark}Asymptoticity of geodesic rays may be seen as a
generalization to arbitrary metric spaces of parallelism of geodesic rays in Euclidean space. Moreover, equivalence classes of asymptotic geodesic rays are the tool to define the visual boundary of a geodesic metric space. It is well known, see for example \cite[Prop. 10.1.4]{Pap},  that two geodesic rays in a geodesic metric space are asymptotic if and
only if their images are at finite Hausdorff distance, a notion defined below.
\end{remark}

\begin{definition}\label{dHaus}
For two geodesic rays $r_{1}$ and $r_{2}$ in a geodesic metric space $(X,d)$, define
their Hausdorff distance by 
\begin{equation}
d_{H}\left( r_{1},r_{2}\right) =\max \left\{ \sup_{x\in \func{Im}%
r_{1}}d\left( x,\func{Im}r_{2}\right) ,\sup_{x\in \func{Im}r_{2}}d\left( x,%
\func{Im}r_{1}\right) \right\} \label{dHauss}
\end{equation}%
where the distance of a point $\alpha$ from a set $B$ is $d\left( \alpha,B\right)
=\inf_{\beta\in B}d\left( \alpha,\beta\right) .$
\end{definition}
As geodesics in $\left( \mathbb{D}^2, D\right)$ coincide with Euclidean lines, it is natural to examine whether the natural Euclidean boundary $\mathbb{S}^1$ of $\mathbb{D}^2$ coincides with the set of equivalence classes of asymptotic geodesic rays in $\left( \mathbb{D}^2, D\right). $ \\
We say that two geodesics rays $r_{1}$ and $r_{2}$ coincide at infinity if, as Euclidean lines, intersect the same point of $\mathbb{S}^1 \equiv\partial \mathbb{D}^2.$  
\begin{theorem}Let $r_{1}$ and $r_{2}$ be two geodesic rays in $\left( \mathbb{D}^2, D\right). $ Then $r_{1}$ and $r_{2}$ coincide at infinity if and only if they are asymptotic. 
\end{theorem}

\begin{proof} The only if direction follows from Lemma \ref{lineINF}.
In view of the above Remark we will show that $r_{1}$ and $r_{2}$ coincide at infinity then their images are at finite Hausdorff distance. 

Since rotations around the origin $O$ are isometries (see Proposition \ref{rotary}) we may assume that the common point at infinity is $%
(1,0)\in \partial \mathbb{D}^{2}.$ We will first examine the case where one
of the geodesic rays is the positive $x-$axis and the other one is contained in the upper
half disk forming an angle $\omega \in \lbrack 0,\pi /2)$ with the $x-$axis
at $(1,0).$ Clearly, the Hausdorff distance is increasing with respect to $\omega .$ Thus, we may consider a geodesic ray being  contained entirely in
the first quadrant forming an angle $\omega \in \lbrack \pi /4,\pi /2)$ with
the $x-$axis at $(1,0).$ The general case will then follow easily.

Set $\lambda =\tan \omega $ and for any $x$ satisfying $0<x<1$ consider the
point $B=(x,0)$ on the (geodesic) $x-$axis and the point $A=\left( x,\lambda
(1-x)\right) $ on the other geodesic ray. Set $\theta _{x}$ to be the angle
formed by the segment $OA$ and the $x-$axis. Clearly the distance $D\left(
A,B\right) $ depends on $x$ and it suffices to show that 
\begin{equation}
\lim_{x\rightarrow 1}D\left( A,B\right) \mathrm{\ is\ bounded.}  \label{limB}
\end{equation}
Recall that for a direction $\Delta _{\theta },$ $\theta \in \left[ 0,\pi %
\right] $ and a point $X$ in the interior of the unit disk we denote by $%
X_{\theta }$ its projection on the diameter $\Delta _{\theta }.$ For the
points $A$ and $B$ we have%
\begin{equation*}
\left\Vert OA_{\theta }\right\Vert =\left\Vert OA\right\Vert \cos \left(
\theta -\theta _{x}\right) =x\frac{\cos \left( \theta -\theta _{x}\right) }{%
\cos \theta _{x}}\mathrm{\ \ and\ \ }\left\Vert OB_{\theta }\right\Vert
=x\cos \theta .
\end{equation*}%
As $A$ is contained in the first quadrant, for all $\theta \in \left[ 0,\frac{%
\pi }{2}\right] $ we obtain%
\begin{equation}
d_{\theta }\left( A,B\right) =d_{I}\left( A_{\theta },B_{\theta }\right) =%
\frac{1}{1-\left\Vert OA_{\theta }\right\Vert }-\frac{1}{1-\left\Vert
OB_{\theta }\right\Vert }=\frac{1}{1-x\frac{\cos \left( \theta -\theta
_{x}\right) }{\cos \theta _{x}}}-\frac{1}{1-x\cos \theta }.  \label{Iab}
\end{equation}%
In a similar manner we find%
\begin{equation*}
d_{\theta }\left( A,B\right) =%
\begin{cases}\displaystyle
\frac{1}{1-x\frac{\cos \left( \theta -\theta _{x}\right) }{\cos \theta _{x}}}%
+\frac{1}{1+x\cos \theta }, 
& \displaystyle
 \text{if }\theta \in \left[ \frac{\pi }{2},%
\frac{\pi }{2}+\theta _{x}\right] \\[8mm] \displaystyle
-\frac{1}{1+x\frac{\cos \left( \theta -\theta _{x}\right) }{\cos \theta _{x}}%
}+\frac{1}{1+x\cos \theta }, & \text{if }\theta \in \left[ \frac{\pi }{2}%
+\theta _{x},\pi \right]%
\end{cases}%
\end{equation*}%
In view of (\ref{limB}) we will only examine the limit as $x\rightarrow 1$
of the integral%
\begin{equation*}
\int_{0}^{\theta _{0}}d_{\theta }\left( A,B\right) d\theta =\int_{0}^{\theta
_{0}}\left[ \frac{1}{1-x\frac{\cos \left( \theta -\theta _{x}\right) }{\cos
\theta _{x}}}-\frac{1}{1-x\cos \theta }\right] d\theta
\end{equation*}%
for sufficiently small $\theta _{0}$ (to be chosen later) because all the
above expressions for $d_{\theta }\left( A,B\right) $ are continuous and
bounded on $\left[ \theta _{0},\pi -\theta _{0}\right] $ and the integral $%
\int_{\pi -\theta _{0}}^{\pi }d_{\theta }\left( A,B\right) d\theta $ is
treated similarly. By substituting%
\begin{equation*}
\sin \theta _{x}=\frac{\lambda \left( 1-x\right) }{\sqrt{x^{2}+\lambda^2\left(
1-x\right) ^{2}}}\mathrm{\ \ and \ \ }\cos \theta _{x}=\frac{x}{\sqrt{%
x^{2}+\lambda^2\left( 1-x\right) ^{2}}}
\end{equation*}%
in (\ref{Iab}) we obtain

\begin{align}
d_{\theta }\left( A,B\right) & =\frac{x\sin \theta _{x}\sin \theta }{\cos
\theta _{x}-x\cos \left( \theta -\theta _{x}\right) -x\cos \theta _{x}\cos
\theta +x^{2}\cos \left( \theta -\theta _{x}\right) \cos \theta }  \notag \\
& =\frac{\lambda \left( 1-x\right) \sin \theta }{1-x\cos \theta -\lambda
\left( 1-x\right) \sin \theta -x\cos \theta +x^{2}\cos ^{2}\theta +\lambda
x\left( 1-x\right) \cos \theta \sin \theta }  \notag  \\
& =\frac{\lambda \left( 1-x\right) \sin \theta }{\left( 1-x\cos \theta
\right) \left( 1-x\cos \theta -\lambda \left( 1-x\right) \sin \theta \right) 
}\label{triple}
\end{align}%
Define 
\begin{equation*}
\Phi \left( \theta \right) =\frac{\sin \theta }{1-x\cos \theta -\lambda
\left( 1-x\right) \sin \theta }
\end{equation*}%
and a straightforward calculation shows that 
\begin{equation*}
\Phi ^{\prime }\left( \theta \right) =\frac{-x+\cos \theta }{\left( 1-x\cos
\theta -\lambda \left( 1-x\right) \sin \theta \right) ^{2}}.
\end{equation*}%
Therefore, there exists a unique angle $\omega _{x}\in \left( \theta _{x},%
\frac{\pi }{2}\right) $ such that 
\begin{equation*}
\cos \omega _{x}=x\Longleftrightarrow \Phi ^{\prime }\left( \theta \right)
=0.
\end{equation*}%
Using the equalities $\cos \omega _{x}=x$ and $\sin \omega _{x}=\sqrt{1-x^{2}%
}$ it is easily shown that%
\begin{align}
\sqrt{1-x}\Phi \left( \omega _{x}\right) & =\frac{\sqrt{1-x}\sin \omega _{x}%
}{1-x\cos \omega _{x}-\lambda \left( 1-x\right) \sin \omega _{x}}=\frac{%
\sqrt{1-x}\sqrt{1-x^{2}}}{1-x^{2}-\lambda \left( 1-x\right) \sqrt{1-x^{2}}}  
\notag \\
& =\frac{\sqrt{1+x}\left( 1-x\right) }{\left( 1-x\right) \left( 1+x-\lambda 
\sqrt{1-x^{2}}\right) }\longrightarrow \frac{\sqrt{2}}{2}\mathrm{\ \ as\ \ }%
x\rightarrow 1.
\end{align}%
The choice of $x$ determines both $A$ and $B$ as well as $\theta _{x}$,
thus, we may choose $\theta _{0}$ such that 
\begin{equation}
\mathrm{for\ \ all\ \ }\theta \leq \theta _{0},\sqrt{1-x}\Phi \left( \omega
_{x}\right) \leq 2.  \label{u0}
\end{equation}%
As the quantity $\Phi \left( \theta \right) $ attains its maximum at $\theta
=\omega _{x}$ we have, using (\ref{triple}),
\begin{align}
\int_{0}^{\theta _{0}}d_{\theta }\left( A,B\right) d\theta &
=\int_{0}^{\theta _{0}}\frac{\lambda \left( 1-x\right) }{\left( 1-x\cos
\theta \right) }\Phi \left( \theta \right) d\theta \leq \int_{0}^{\theta
_{0}}\frac{\lambda \left( 1-x\right) }{\left( 1-x\cos \theta \right) }\Phi
\left( \omega _{x}\right) d\theta   \notag \\
& =\int_{0}^{\theta _{0}}\frac{\lambda \sqrt{1-x}}{\left( 1-x\cos \theta
\right) }\sqrt{1-x}\Phi \left( \omega _{x}\right) d\theta \leq
2\int_{0}^{\theta _{0}}\frac{\lambda \sqrt{1-x}}{\left( 1-x\cos \theta
\right) }d\theta 
\end{align} 
where the latter inequality follows from (\ref{u0}). It suffices to show
that $\int_{0}^{\theta _{0}}\frac{\sqrt{1-x}}{\left( 1-x\cos \theta \right) }%
d\theta $ is bounded which follows from the following identity 
\begin{equation*}
\int \frac{1}{1-x\cos \theta }d\theta =\frac{2}{\sqrt{1-x^{2}}}\tan
^{-1}\left( \sqrt{\frac{1+x}{1-x}}\tan \frac{\theta }{2}\right) 
\end{equation*}%
and the observation that the range of the inverse tangent function is a
bounded interval:
\begin{align}
\int_{0}^{\theta _{0}}\frac{\sqrt{1-x}}{\left( 1-x\cos \theta \right) }%
d\theta & =\frac{\sqrt{1-x}}{\sqrt{1-x^{2}}}\left[ \tan ^{-1}\left( \sqrt{%
\frac{1+x}{1-x}}\tan \frac{\theta }{2}\right) \right] _{\theta =0}^{\theta
=\theta _{0}}    \notag \\
&=\frac{1}{\sqrt{1+x}}\tan ^{-1}\left( \sqrt{\frac{1+x}{1-x}}%
\tan \frac{\theta _{0}}{2}\right) .   \notag
\end{align}
We now discuss the case of two arbitrary asymptotic geodesic rays $r_{1}$
and $r_{2}.$ As mentioned at the beginning of the proof, we may assume that
the  common boundary point is $(1,0)\in \partial \mathbb{D}^{2}.$ Let $%
\omega _{i}\in \left( -\pi /2,\pi /2\right) ,$ $i=1,2$ be the angle formed
by $r_{i}$ and the $x-$axis. Denote by $r_{x}$ the geodesic ray whose image is
the positive $x-$axis in the unit disk. If both $\omega _{1},\omega _{2}$ are
positive and, say, $\omega _{1}<\omega _{2}$ then $d_{H}\left(
r_{1},r_{2}\right) <d_{H}\left( r_{x},r_{2}\right) $. If $\omega _{1}\omega
_{2}<0$ then a triangle inequality argument asserts that $d_{H}\left(
r_{1},r_{2}\right) \leq d_{H}\left( r_{x},r_{1}\right) +d_{H}\left(
r_{x},r_{1}\right) .$ This completes the proof of the theorem. \end{proof}

\noindent {\bf Acknowledgments.} The authors would like to thank the anonymous referee for very helpful comments and suggestions which, among other things, improved the exposition of the proof of Theorem \ref{eight}.

\noindent Conflict of Interest statement: On behalf of all authors, the corresponding author states that there is no conflict of
interest.\\
\noindent Data Availability Statement: Data sharing not applicable to this article as no datasets were generated or analysed
during the current study.


\begin{thebibliography}{99}   
\bibitem{Beltrami} E. Beltrami, Saggio di Interpretazione della geometria
non-Euclidea. Giornaledi Matematiche vol. VI (1868). Beltrami's Works, Vol.
I, pp. 374--405. French translation: Essai d'interpr\'{e}tation de la g\'{e}%
om\'{e}trie noneucilienne, by J. Ho\"{u}el, Annales ENS, 1re S\'{e}r. t. 6
(1869), p. 251-288. English translation: Essay on the interpretation of
non-Euclidean geometry, by J. Stillwell, in \cite{Stilwell} pp. 7--34.

\bibitem{Buseman} H. Busemann, \textit{On Hilbert fourth problem,} Uspechi Mat. Nauk
21 (1966), no. 1 (127), 155-164.

\bibitem{BusKel}H. Busemann and P.J. Kelly, \textit{Projective geometry and projective metrics}, Academic Press, New York 1953.

\bibitem{Klein} F. Klein, Uber die sogenannte Nicht-Euklidische Geometrie
(erster Aufsatz). Math. Ann. IV (1871), 573--625. English translation by J.
Stillwell in \cite{Stilwell}.

\bibitem {Pap}A. Papadopoulos, \textit{Metric Spaces, Convexity and Nonpositive Curvature,} IRMA Lectures in Mathematics and Theoretical Physics 6, European Mathematical Society, 2013.

\bibitem{Papadopoulos} A. Papadopoulos, Hilbert's fourth problem, Hanbook of
Hilbert Geometry, Irma Lectures in Mathematics and Theoretical Physics 22,
A. Papadopoulos and M. Troyanov editors, EMS 2014.

\bibitem{Pogorelov1} A. V. Pogorelov, Hilbert's fourth problem. Translated
by R. A. Silverman, Edited by I. Kra, in cooperation with E. Zaustinskyi, V.
H. Winston \& Sons, Washington D. C., 1979. Original by Nauka.

\bibitem{Pogorelov2} A. V. Pogorelov, A complete solution of Hilbert's
fourth problem, Sov. Math. 14 (1973), pp. 46--49.

\bibitem{Stilwell} J. Stillwell, Sources of hyperbolic geometry, History of
Mathematics series, Vol. 10, AMS-LMS, 1996.

\end{thebibliography}
\end{document}